\documentclass[10pt]{amsart}

\usepackage{amsfonts, amsmath, amsthm, amssymb, bbm}

\newtheorem{theorem}{Theorem}

\newtheorem{corollary}[theorem]{Corollary}

\theoremstyle{definition}

\newtheorem{fact}{Fact}
\newtheorem{claim}{Claim}

\theoremstyle{remark}

\newtheorem{remark}[theorem]{Remark}

\def\cf{{\mathrm{cf} \, }}

\def\ot{{\mathrm{ot} }}
\def\Lim{{\mathrm{Lim} }}
\def\rest{{\upharpoonright}}

\begin{document}

\title{The effect of forcing axioms on the tightness of the $G_\delta$-modification}

\author{William Chen-Mertens}
\address{York University}
\email{chenwb@gmail.com}

\author{Paul J. Szeptycki}
\address{York University}
\email{szeptyck@yorku.ca}

\begin{abstract}
We show that  $\mathsf{PFA}$ implies that the tightness $t(X_\delta)$ of the $G_\delta$-modification of a Fr\'echet $\alpha_1$-space $X$ is at most $\omega_1$, while $\Box(\kappa)$ implies that there is a Fr\'echet $\alpha_1$-space with $G_\delta$-tightness equal to $\kappa$. We use the example constructed from $\Box(\kappa)$ to show that a local version of the bound $t(X_\delta)\le 2^{t(X)L(X)}$ does not hold. We also construct, assuming $\mathsf{MA}$, an example of a Fr\'echet space whose $G_\delta$-tightness is larger than $\omega_1$.
\end{abstract}

\maketitle

\section{Introduction}
Given a topological space $X$, we can form the $G_\delta$-modification, denoted $X_\delta$, by taking the $G_\delta$ sets of $X$ to be a base. This is interesting for many reasons; for example, $X_\delta$ is easily seen to be a $P$-space.

Bella and Spadaro \cite{017BS0} investigated the relationships between the cardinal invariants of a space $X$ and those of $X_\delta$. In particular, recall that the tightness $t(x,X)$ of $x\in X$ is the minimum $\kappa$ so that for every $A\subseteq X$ so that $x\in\overline{A}$, there is $B\subseteq A$ with $|B|\le \kappa$ so that $x\in \overline{B}$, and the tightness $t(X)$ of a space $X$ is the supremum of the tightness of all points of $X$. Dow et al. \cite{018DJSSW0} obtained many results on the tightness left open by the earlier work. In particular:
\begin{enumerate}
\item $t(X_\delta)\le 2^{t(X)}$ if $X$ is Lindel\"of.
\item if there is a nonreflecting stationary set of countable cofinality ordinals in $\kappa$, then there is a space $X$ so that $t(X)=\omega$ (in fact, $X$ is Fr\'echet) but $t(X_\delta)=\kappa$. 
\item if $\kappa$ is strongly compact and $t(X)<\kappa$, then $t(X_\delta)<\kappa$.
\end{enumerate}
They asked whether there is a $\mathsf{ZFC}$ example of a space $X$ so that $t(X)=\omega$ but $t(X_\delta)>2^\omega$.

This was answered positively by Usuba in \cite{018U0}. His example was countably tight but not Fr\'echet, and he asked whether there is a Fr\'echet example.

We show that, assuming $\mathsf{PFA}$, there is no Fr\'echet $\alpha_1$ example in $\mathsf{ZFC}$---in fact, we show that  $\mathsf{PFA}$ implies $t(X_\delta)\le \omega_1$ for $X$ a Fr\'echet $\alpha_1$-space. On the other hand, in the presence of $\Box(\kappa)$, there is a Fr\'echet $\alpha_1$-space $X$ having $t(X_\delta)=\kappa$. 

By increasing the strength of our assumptions from $\mathsf{PFA}$ to $\mathsf{MM}$, we do not get the analogous result for Fr\'echet spaces generally; in fact, we construct a Fr\'echet space whose tightness in the $G_\delta$-modification is $\mathfrak{c}>\omega_1$ from hypotheses much weaker than $\mathsf{MM}$.

Finally, we use the example constructed from $\Box(\kappa)$ to show that a local version of the bound $t(X_\delta)\le 2^{t(X)L(X)}$ does not hold, answering a question of Bella and Spadaro.

\section{Fr\'echet $\alpha_1$-spaces}

Let $X$ be a space. A point $x\in X$ is Fr\'echet if for every subset $A$, $x\in \overline{A}$ if and only if there is a subsequence $\{x_n:n<\omega\}\subseteq A$ converging to $x$, and $\alpha_1$ if whenever there are countably many sequences converging to $x$, then there is another converging sequence which, modulo finite sets, contains each of them. The space $X$ is Fr\'echet ($\alpha_1$) if every point is Fr\'echet ($\alpha_1$). The $\alpha_1$ property was introduced in \cite{979A0} to study the productivity of the Fr\'echet property.

The example in \cite{018DJSSW0} of a Fr\'echet space whose tightness in the $G_\delta$ topology is large is constructed from a non-reflecting stationary set of ordinals of countable cofinality. This assumption is consistent with $\mathsf{PFA}$, however, $\mathsf{PFA}$ suffices to give a bound with the additional hypothesis of being an $\alpha_1$-space.

A $P$-ideal on $X$ is an ideal $\mathcal{I}$ consisting of countable subsets of $X$ which contains all of its finite subsets and moreover has the property that for any $\{A_n:n<\omega\}\subseteq \mathcal{I}$, there is some $A\in \mathcal{I}$ so that $A_n\setminus A$ is finite for all $n$. The $P$-ideal dichotomy, or $\mathsf{PID}$, is a consequence of $\mathsf{PFA}$ that states that for every $P$-ideal $\mathcal{I}$ on any set $X$, either
\begin{itemize}
\item there is $Y\subseteq X$ uncountable so that every countable subset of $Y$ is in $\mathcal{I}$, or
\item $X$ can be written as the countable union of sets $Y_n$, each containing no infinite subset in the ideal.
\end{itemize}

\begin{theorem}\label{thm:pid}
$\mathsf{PID}$ implies that for any Fr\'echet $\alpha_1$-space $X$, $t(X_\delta)\le \omega_1$.
\end{theorem}
\begin{proof}
Suppose $x\in \overline{A}^\delta$ for some $x\in X$ and $A\subseteq X$ not containing $x$. Consider the ideal $\mathcal{I}$ of subsets of $A$ which are either finite or converge to $x$. Since $X$ is an $\alpha_1$-space, $\mathcal{I}$ is a $P$-ideal.

In the first case of the dichotomy, there is $B\subseteq A$ with $|B|=\omega_1$ such that every countable subset is in $\mathcal{I}$. So for any open $U$ containing $x$, $B\setminus U$ is finite. Therefore, $x\in \overline{B}^\delta$.

In the second case of the dichotomy, $A=\bigcup_{n<\omega} A_n$, where $A_n$ is orthogonal to $\mathcal{I}$. It follows that $x\not\in \overline{A}_n$, otherwise---as $X$ is Fr\'echet---there is some sequence $\{x_i:i<\omega\}\subseteq A_n$ converging to $x$, and then this would be in $\mathcal{I}$. So $x\not\in \overline{A}^\delta$, a contradiction.
\end{proof}
\begin{remark}
In fact, the proof was local, i.e., it showed that under $\mathsf{PID}$, for any Fr\'echet $\alpha_1$-point $x$, $t(x,X_\delta)\le \omega_1$. Moreover, for any $A$ with $x\in \overline{A}^\delta$ there is an $\omega_1$-sequence in $A$ converging to $x$ in $X_\delta$.

\end{remark}

We now show that the conclusion of Theorem \ref{thm:pid} can fail consistently 
by providing an example of a Fr\'echet $\alpha_1$-space $X$ whose tightness in the $G_\delta$ topology is large. The example starts from the assumption $\Box(\kappa)$, namely that there exists a sequence $\langle C_\alpha:\alpha<\kappa\rangle$ so that:
\begin{itemize}
\item $C_\alpha$ is club in $\alpha$ for $\alpha$ limit,
\item $C_{\alpha+1}=\{\alpha\}$,
\item $C_\alpha\cap \beta=C_\beta$ for every $\beta$ which is a limit point of $C_\alpha$,
\item there is no club $C$ in $\kappa$ so that $C\cap \alpha=C_\alpha$ for every $\alpha$ which is a limit point of $C$ (such a club is called a \emph{thread}).
\end{itemize}

In the process of ``walking'' from $\beta$ to $\alpha$ along the sequence, starting at $\beta_0=\beta$ and at each step taking $\beta_{i+1}=\min(C_{\beta_i}\setminus \alpha)$, we will eventually reach $\beta_n=\alpha$ as the $\beta_i$ are a decreasing sequence of ordinals. Let $\mathrm{Tr}(\alpha,\beta)=\{\beta_i:0\le i\le n\}$, and $\rho_2(\alpha,\beta)=n$, more precisely defined inductively so that $\rho_2(\alpha,\alpha)=0$ and $\rho_2(\alpha,\beta)=\rho_2(\alpha,\min(C_\beta\setminus\alpha)+1$.

We will use the following basic facts:
\begin{fact}[\cite{010T0}]\label{fact:walks}
 $\,$
\begin{enumerate}
	\item For $\alpha<\beta<\kappa$,
$$\sup_{\xi<\alpha}|\rho_2(\xi,\alpha)-\rho_2(\xi,\beta)|<\omega.$$ 
	\item For every family $\mathcal{F}\subseteq [\kappa]^{<\omega}$ of pairwise disjoint finite subsets of $\kappa$ with $|\mathcal{F}|=\kappa$ and for every integer $n$ there exist $a$, $b$ both in $\mathcal{F}$
such that $$\rho_2(\alpha, \beta) > n$$ for all $\alpha\in a$, $\beta\in b$.
\end{enumerate}
\end{fact}
Part (2) of Fact \ref{fact:walks} can be shown to be equivalent to the nonexistence of a thread through the sequence. However, we will only need the following direct consequence, obtained from the case where $\mathcal{F}$ consists of singletons or of pairs:

\begin{itemize}
    \item [(2')] For any $A,B\subseteq \kappa$ unbounded and any $n<\omega$, there are $\alpha<\beta$ so that $\alpha\in A$, $\beta\in B$, and $\rho_2(\alpha,\beta)>n$. In particular, for any $A\subseteq \kappa$ unbounded and any $n<\omega$, there are $\alpha<\beta$ both in $A$ so that $\rho_2(\alpha,\beta)>n$.
\end{itemize}

\begin{theorem}\label{thm:example}
Assume $\kappa$ is regular and $\Box(\kappa)$ holds. Then there is a Fr\'echet $\alpha_1$-space $X$ so that $t(X_\delta)=\kappa$.
\end{theorem}
\begin{proof}
Let $\langle C_\alpha:\alpha<\kappa\rangle$ be the $\Box(\kappa)$ sequence.

Then we define a topology on $X=\kappa+1$ so that the points of $\kappa$ are isolated, and the open neighborhoods for $\kappa$ are generated by sets of the form
$$\{\kappa\}\cup \bigcup_{\alpha\in \Lim(\kappa)}\{\xi<\alpha:\rho_2(\xi,\alpha)>n_\alpha\}$$
where $n_\alpha<\omega$.

Equivalently, using part (1) of Fact \ref{fact:walks}, the open neighborhoods are generated by sets of the form
$$\{\kappa\}\cup \kappa\setminus\{\xi<\alpha:\rho_2(\xi,\alpha)\le n\}$$
where $\alpha<\kappa$ and $n<\omega$.

With this topology, the ideal of converging sequences is exactly the $P$-ideal
$$\mathcal{I}=\{A\subseteq \kappa: (\forall \alpha<\kappa) (\rho_2)_\alpha \textrm{ is finite-to-one on }A \}$$
used by Todorcevic \cite{000T0} to show that $\mathsf{PID}$ refutes $\Box(\kappa)$ for all regular uncountable $\kappa$.

It is easy to check that this defines a Hausdorff topology on $\kappa+1$.

\begin{claim}
$X$ is Fr\'echet.
\end{claim}
Suppose that $\kappa\in \overline{A}$. Then there must be some $\alpha<\kappa$ so that $\sup_{\xi\in A\cap \alpha}\rho_2(\xi,\alpha)=\omega$, otherwise for each $\alpha$ we can take $n_\alpha=\sup_{\xi<\alpha}\rho_2(\xi,\alpha)$ and define an open neighborhood of $\kappa$ disjoint from $A$. For each $n<\omega$ let $\xi_n\in A\cap\alpha$ be so that $\rho_2(\xi_n,\alpha)\ge n$. Now it is straightforward using Fact \ref{fact:walks} (1) to show that $\{\xi_n:n<\omega\}$ converges to $x$.

\begin{claim}
$X$ is an $\alpha_1$-space.
\end{claim}

This claim follows from the proof that $\mathcal{I}$ is a $P$-ideal: suppose that $A_n, n<\omega$, are converging sequences. Fix a limit ordinal $\alpha\ge \sup(\bigcup_n A_n)$. Let $B_n=A_n\setminus \{\xi:\rho_2(\xi,\alpha)\le n\}$, so $B =\bigcup_n B_n$ almost contains each $A_n$. $B$ is itself a converging sequence, since for each $n$ the set $\{\xi\in B:\rho_2(\xi,\alpha)\le n\}$ is contained in $\bigcup_{m\le n}\{\xi\in B_m:\rho_2(\xi,\alpha)\le n\}$, a finite union of finite sets.

\begin{claim}\label{unbddclaim}
Suppose $B\subseteq \kappa$ is unbounded. Then $\kappa\in \overline{B}$.
\end{claim}
Suppose not, so for every $\alpha\in\Lim(\kappa)$ there is $n_\alpha<\omega$ so that
$$B\cap \bigcup_{\alpha\in \Lim(\kappa)}\{\xi<\alpha: \rho_2(\xi,\alpha)>n_\alpha \}=\emptyset.$$
There is an unbounded $D\subseteq \Lim(\kappa)$ and $n<\omega$ so that $n_\alpha=n$ for all $\alpha\in D$. By Fact \ref{fact:walks} (2), there are $\alpha\in B$ and $\beta\in D$ so that $\alpha<\beta$ and $\rho_2(\alpha,\beta)>n$, a contradiction.

By Claim \ref{unbddclaim}, every open neighborhood of the point $\kappa$ is a co-bounded subset of $\kappa$. Since $\cf(\kappa)>\omega$, we immediately have:

\begin{claim}
$\kappa\in \overline{\kappa}^\delta.$
\end{claim}

Now it remains to prove

\begin{claim}
$t(X_\delta)=\kappa$.
\end{claim}
Suppose otherwise. Then there is some bounded $B\subseteq \kappa$ so that $\kappa\in\overline{B}^\delta$. Let $\beta\ge \sup B$ be a limit ordinal.

Now let $n<\omega$. By Fact \ref{fact:walks} each $\alpha<\kappa$ there is $n_\alpha$ so that $\rho_2(\xi,\alpha)>n_\alpha$ implies $\rho_2(\xi,\beta)>n$. Define 
$$U_n=\{\kappa\}\cup \bigcup_{\alpha\in \Lim(\kappa)}\{\xi<\alpha:\rho_2(\xi,\alpha)>n_\alpha\}$$
so that $U_n$ is open and disjoint from $\{\xi\in B:\rho_2(\xi,\beta)=n\}$. So $\bigcap_n U_n$ is a $G_\delta$ set disjoint from $B$. 

This completes the proof of the theorem.

\end{proof}

\section{A Fr\'echet space having $G_\delta$-tightness $\omega_2$}
The principle $\mathsf{MM}$ refutes all of the examples of Fr\'echet spaces constructed thus far with $G_\delta$-tightness larger than $\omega_1$. In light of Theorem \ref{thm:pid}, it is natural to conjecture that $\mathsf{MM}$ implies that all Fr\'echet spaces have $G_\delta$-tightness at most $\omega_1$. Surprisingly, though, it implies the opposite.

\begin{theorem}
Assume $\mathsf{MA}(\omega_1)$ and $2^{\omega_1}=\omega_2$. Then there is a Fr\'echet topology on $X=\omega_2+1$ such that $t(X_\delta)=\omega_2$.
\end{theorem}
\begin{proof}
By a classical result of Gregory \cite{976G0}, the cardinal arithmetic assumption implies that $\diamondsuit_{E^2_0}$ holds. From the diamond sequence, we can define a ladder system $\{S_\alpha:\alpha\in E^2_0\}$ so that:
\begin{itemize}
\item $\ot(S_\alpha)=\omega$,
\item $\sup S_\alpha=\alpha$,
\item for any unbounded $X\subseteq \omega_2$, the set $\{\alpha\in E^2_0:S_\alpha\subseteq X\}$ is stationary.
\end{itemize}
Let all the points in $\omega_2$ be isolated. The open neighborhoods of $\omega_2$ are given by sets of the form $\{\omega_2\}\cup \bigcup_{\alpha\in E^2_0} T_\alpha,$
where $S_\alpha\setminus T_\alpha$ is finite.

This defines a Fr\'echet topology, since for any set $A$ with $\omega_2\in \overline{A}$, there is some $\alpha$ so that $S_\alpha\cap A$ is infinite (and thus of order-type $\omega$), and then $S_\alpha\cap A$ converges to $\omega_2$.

Now $\omega_2\in \overline{\omega_2}^\delta$, otherwise $\{\omega_2\}$ is $G_\delta$ and write $\{\omega_2\}=\bigcap_{n<\omega} U_n$, where each $U_n$ is an open neighborhood of $\omega_2$. Define $A_n=\{\xi<\omega_2:n \textrm{ is least such that } \xi\not\in U_n\}$, so that $A_n$ partitions $\omega_2$. Now there is $n$ so that $A_n$ is unbounded in $\omega_2$, and therefore by the property of $\{S_\alpha:\alpha\in E^2_0\}$ obtained from $\diamondsuit_{E^2_0}$ there exists $\alpha\in E^2_0$ so that $S_\alpha\subseteq A_n$. Now $S_\alpha\cap U_n=\emptyset$, which is impossible as $U_n$ is open.

Finally we prove that $t(X_\delta)>\omega_1$. For $\gamma<\omega_2$, we find a $G_\delta$ subset containing the point $\omega_2$ but disjoint from $\gamma$. For this, it is enough to find a function $F:\gamma\rightarrow \omega$ so that $F\rest S_\alpha$ is finite-to-one for all $\alpha< \gamma$.

From  $\mathsf{MA}(\omega_1)$, we use only the following consequence proven by Devlin and Shelah \cite{978DS0}, known as ladder system uniformization:
\begin{itemize}
\item If $\langle S_\alpha:\alpha\in\Lim(\omega_1)\rangle$ is a \emph{ladder system} (so that $S_\alpha$ has order-type $\omega$ and is cofinal in $\alpha$), and $f_\alpha:S_\alpha\rightarrow\omega$, then there is $f:\omega_1\rightarrow \omega$ so that $f\rest \alpha$ agrees with $f_\alpha$ on all but finitely many members of $S_\alpha$.
\end{itemize}

We proceed to construct $F$ by induction on $\gamma$. Fix a bijection $\varphi:\omega\times\omega\rightarrow\omega$.

For $\alpha\in\Lim(\omega_1)$, let $f_\alpha:\alpha\rightarrow\omega$ be the increasing enumeration of $S_\alpha$, extended arbitrarily to points outside of $S_\alpha$. Then the uniformizing function $f$ is finite-to-one on all $S_\alpha$. 

If $\gamma=\beta+1$ for some $\beta$, then let $F_\beta$ be the function constructed at $\beta$ and let $F(\xi)=\varphi(F_\beta(\xi),f_\alpha(\xi))$. If $\gamma$ is of the form $\delta+\omega$, then take the function constructed at stage $\delta+1$ and extend it arbitrarily. 

Otherwise, take a club $E\subseteq \gamma$ of minimum possible order-type consisting of limit ordinals. Using the induction hypothesis, for each $\beta\in E$ there is $F_\beta:\beta\rightarrow \omega$ so that $F_\beta\rest S_\alpha$ is finite-to-one for all $\alpha< \beta$. Furthermore, using ladder system uniformization in case $\cf(\gamma)=\omega_1$, there is a function $F':\bigcup_{\beta\in E} S_\beta\rightarrow \omega$ which is finite-to-one on every $S_\beta$, $\beta\in E$ (noting that $\bigcup_{\beta\in E} S_\beta$ is a homeomorphic copy of $\omega_1$). Extend $F'$ arbitrarily to $\gamma$. Finally, let $F(\xi)=\varphi(F_\beta(\xi),F'(\xi))$, where $\beta$ is the minimum of $E\setminus (\xi+1)$.

\end{proof}

\section{Countably tight points in compact spaces}
As mentioned in the introduction, it was shown in \cite{018DJSSW0} that a compact Hausdorff space with countable tightness has tightness at most $\mathfrak{c}$ in the $G_\delta$-modification. But their proof was not local---showing that a point had $G_\delta$-tightness $\le \mathfrak{c}$ used the countable tightness of other points in the space.

It is natural to ask whether the local version of the result indeed holds, that is, whether a point of countable tightness in a compact space $X$ has tightness at most $\mathfrak{c}$ in the $G_\delta$-modification. We can use the example of Theorem \ref{thm:example} to answer this question negatively by constructing a compact Hausdorff space having a point of countable tightness whose tightness becomes large in the $G_\delta$-modification.

The construction is to take the \v{C}ech--Stone compactification $\beta X$ of the space of Theorem \ref{thm:example}. In the language of \cite{979A0}, the \emph{absolute tightness} of a point $x\in X$ is its tightness in any compactification (and there it is proven that this does not depend on the choice of compactification). We will show that the absolute tightness of $\kappa\in X$ is countable.

\begin{theorem}\label{thm:localexample}
Let $X$ be the space constructed in Theorem \ref{thm:example}. Then in the \v{C}ech--Stone compactification $\beta X$ (and hence in any compactification of $X$) the point $\kappa$ is a Fr\'echet point.
\end{theorem} 
\begin{proof}
Recall that the points of $\beta X$ are the maximal filters of closed subsets of $X$, and the topology consists of all sets of the form $\{F\in \beta X : (X\setminus U)\not\in F  \}$, where $U$ is open in $X$.  $X$ embeds into $\beta X$, where each $x\in X$ is mapped to the principal ultrafilter on $x$.

Suppose $\kappa\in \overline{A}$ for $A\subseteq \beta X$. We can assume that $A\subseteq \beta X\setminus X$.

\begin{claim}
$\,$
\begin{itemize}
\item For each $F\in \beta X\setminus X$, let $\alpha(F)$ be the least ordinal $\alpha$ so that $\alpha\in F$. Then $\alpha(F)$ is a limit ordinal less than $\kappa$.

\item For $F\in A$ and $\beta\in \kappa\setminus\alpha(F)$, there is $n$ so that $$\{\xi<\alpha(F):\rho_2(\xi,\beta)\le n\}\in F.$$
\end{itemize}
\end{claim}
Since $F$ is non-principal, $\alpha(F)$ is a limit ordinal and there is $B\in F$ not containing the point $\kappa$. By Claim \ref{unbddclaim}, $B$ is bounded in $\kappa$. Furthermore, there must be some $n$ so that $B\cap \{\xi<\beta:\rho_2(\xi,\beta)\le n\}\in F$, proving the present claim.

For $F\in A$ and $\beta\in \kappa\setminus\alpha(F)$, let
$n_{\beta,F}$ be the unique $n$ so that $\{\xi<\alpha(F):\rho_2(\xi,\beta)= n\}\in F$.

If there is $\beta\in \Lim(\kappa)$ so that 
$\sup\{ n_{\beta,F}:\alpha(F)<\beta\}=\omega,$
then for each $n<\omega$ choose $F_n\in A$ so that $n_{\beta,F_n}>n$. In this case, $\{F_n:n<\omega\}$ converges to $\kappa$.

So we may assume that
$$n_\beta:=\sup\{ n_{\beta,F}:\alpha(F)<\beta\}<\omega,$$
and from this we will derive a contradiction.

In this case, $\{\alpha(F):F\in A\}$ is unbounded, otherwise if $\beta<\kappa$ is an upper bound, we can use Fact \ref{fact:walks} (1) together with our assumption to construct an open neighborhood of $\kappa$ disjoint from $A$.
Take an unbounded set $B\subseteq \Lim(\kappa)$ and $n<\omega$ so that $n_\beta=n$ for all $\beta\in B$. By Fact \ref{fact:walks} (2), there are $F\in A$ and $\beta\in B$ so that $\alpha(F)<\beta$ and $\rho_2(\alpha(F),\beta)>n$.

\begin{claim}
There is $\eta<\alpha(F)$ so that $\rho_2(\xi,\beta)> n$ for all $\xi\in (\eta, \alpha)$.
\end{claim}
Let $\mathrm{Tr}(\alpha(F),\beta)=\{\beta_0,\ldots,\beta_{n+1}\}$. Take $\eta=\sup_{i<n} C_{\beta_i}\cap \alpha$. Since $\alpha\not\in C_{\beta_i}$ for $i<n$, we have $\eta<\alpha$. For any $\xi\in (\eta,\alpha)$, the walk from $\beta$ to $\xi$ goes through $\{\beta_0,\ldots,\beta_n\}$, so must take at least $n+1$ steps, proving the claim.

By minimality of $\alpha(F)$, $\eta+1\not\in F$, so $(\eta,\alpha(F))\in F$ and therefore
$$\{\xi<\alpha(F):\rho_2(\xi,\beta)\ge n+1\}\in F.$$
But $n+1>n_\beta \ge n_{\beta,F}$, contradicting the definition of $n_{\beta,F}$.
\end{proof}

\begin{remark}
In fact, $\kappa$ is a Fr\'echet $\alpha_1$ point in $\beta X$.
\end{remark}

The tightness of a point does not increase in subspaces. Therefore:
\begin{corollary}
If $\Box(\kappa)$ holds, there is a compact space $X$ and $x\in X$ a Fr\'echet point so that $t(x,X_\delta)=\kappa$.
\end{corollary}

\begin{remark}
If $X$ embeds the sequential fan $S(\omega)$, then it is easy to show that the base point cannot have countable tightness in any compactification. Therefore the example in \cite{018DJSSW0} does not satisfy the conclusions of Theorem \ref{thm:localexample}.
\end{remark}

\bibliography{billtop}{}
\bibliographystyle{plain}

\end{document}